\documentclass[aop,MSNbibl,dvips]{arximspdf}

\doi{10.1214/12-AOP748} 
\volume{41}
\issue{5}
\pubyear{2013}
\firstpage{3582}
\lastpage{3605}

\makeatletter

\newtheorem{Claim}{Claim}

\newtheorem{theo}{Theorem}[section]
\newtheorem{lem}[theo]{Lemma}
\newtheorem{cor}[theo]{Corollary}
\newtheorem{prop}[theo]{Proposition}

\renewcommand{\a}{\alpha}
\renewcommand{\b}{\beta}

\newcommand{\vt}{\vartheta}
\newcommand{\D}{\Delta}
\newcommand{\G}{\Gamma}
\renewcommand{\k}{\kappa}
\newcommand{\DD}{\mathcal D}
\renewcommand{\l}{\lambda}
\newcommand{\EE}{\mathcal E}
\newcommand{\s}{\sigma}
\newcommand{\Si}{\Sigma}
\newcommand{\R}{\mathbb R}
\newcommand{\N}{{\mathbb N}}
\newcommand{\Z}{\mathbb Z}

\renewcommand{\P}{\mathbb P}

\renewcommand{\v}{\varphi}
\newcommand{\Om}{\Omega}
\newcommand{\CC}{\mathcal C}
\renewcommand{\O}{\mathcal O}

\newcommand{\LL}{{{\mathcal L}}}
\newcommand{\M}{\mathcal{P}}

\makeatother

\begin{document}
\begin{frontmatter}

\title{Regularity of the entropy for random walks on hyperbolic
groups\thanksref{T1}}
\runtitle{Entropy for hyperbolic groups}

\thankstext{T1}{Supported in part by NSF Grant DMS-08-11127.}

\begin{aug}
\author[A]{\fnms{Fran\c cois} \snm{Ledrappier}\corref{}\ead[label=e1]{fledrapp@nd.edu}}
\runauthor{F. Ledrappier}
\affiliation{University of Notre Dame and Universit\'e Paris 6}
\address[A]{Department of Mathematics\\
University of Notre Dame\\
Notre Dame, Indiana 46656\\
USA\\
and\\
LPMA, UMR CNRS 7599\\
Universit\'e Paris 6\\
Bo\^ite Courrier 188\\
4, Place Jussieu\\
75252 Paris Cedex 05\\
France\\
\printead{e1}} 
\end{aug}

\received{\smonth{10} \syear{2011}}
\revised{\smonth{2} \syear{2012}}

%
\begin{abstract}
We consider nondegenerate, finitely supported random walks on a
finitely generated Gromov hyperbolic group. We show that the entropy
and the escape rate are Lipschitz functions of the probability if the
support remains constant.
\end{abstract}

%

\end{frontmatter}

\section{Introduction}

This paper is an extension of~\cite{L2} to finitely generated Gromov
hyperbolic groups; see~\cite{GH} and Section~\ref{secpreliminaries}
below for the definition of hyperbolic groups. Let $p$ be a finitely
supported\vspace*{1pt} probability measure on an infinite group $G$, and define inductively,
with $ p^{(0)} $ being the Dirac measure at the identity $e$,
\[
p^{(n)} (x) = \bigl[p^{(n-1)} \star p\bigr] (x) = \sum
_{y \in G} p^{(n-1)} \bigl(xy^{-1}\bigr) p(y).
\]
Define the entropy $h_p$ and the escape rate $\ell_p^S $ by
\[
h_p:= \lim_n -\frac{1}{n} \sum
_{x\in G} p^{(n)}(x) \ln p^{(n)}(x),\qquad
\ell_p^S:= \lim_n \frac{1}{n} \sum
_{x\in G} |x| p^{(n)}(x),
\]
where \mbox{$|\cdot|$} is the word metric defined by some symmetric generating set~$S$.
The entropy $h_p$ was introduced by Avez~\cite{Av} and is related to
bounded solutions of the equation on $G$ $f(x) = \sum_{y \in G} f(x y
) p(y)$; see, for example,~\cite{KV}. Erschler and Kaimanovich have
shown that, on Gromov hyperbolic groups, the entropy and the escape
rate depend continuously on the probability $p$ with finite first
moment~\cite{EK}. Here we are looking for a stronger regularity on a
more restricted family of probability measures.
We fix a finite set $F \subset G$ such that $\bigcup_n F^n = G$, and we
consider probability measures in $\M(F)$, where $\M(F)$ is the set of
probability measures $p$ such that $p(x) >0$ if, and only if, $x \in
F$. The set $\M(F)$ is naturally identified with an open subset of the
probabilities on $F$, which is a contractible open polygonal bounded
convex domain in $\R^{|F|- 1}$. We show:
%
%
\begin{theo}\label{main} Assume $G$ is a Gromov hyperbolic group, and
$F$ is a finite subset of $G$ such that $\bigcup_n F^n = G$. Then, with
the above notation, the functions $p \mapsto h_p$ and $p \mapsto\ell
_p^S$ are Lipschitz continuous on $\M(F)$.
\end{theo}

If the infinite hyperbolic group $G$ is amenable, $G$ is virtually
cyclic, and the entropy is vanishing on $\M(F)$. Moreover, it follows
from the formula in~\cite{KL} that the escape rate is Lipschitz
continuous in on $\M(F) $; see the remark after Formula (\ref{escape})
below. If $G$ is a non-Abelian free group, and $F$ a general finite
generating set, then $p \mapsto h_p$ is real analytic (\cite{L2},
Theorem 1.1) and $p \mapsto\ell_p^S $ as well~\cite{Gi}. This
holds more generally for free products; see~\cite{Gi} and~\cite{Gi2}
for the precise conditions. A general nonamenable hyperbolic group has
many common geometric features with non-Abelian free groups, and our
proof follows the scheme of~\cite{L2}. For Gromov hyperbolic groups,
Ancona~\cite{An} proved that the Martin boundary of the random walk
directed by the probability $p$ is the Gromov geometric boundary. Let
$K_\xi(x) $ be the Martin kernel associated to a point $\xi$ of the
geometric boundary. Our main technical result, Proposition
\ref{LipMartin}, uses the description of the Martin kernel by Ancona
(see also~\cite{W,INO}) to prove that $\ln K_\xi(x) $ is a
Lipschitz continuous function of $p$ as a H\"older continuous function
on the geometric boundary. Then, like in~\cite{L2}, we can express
$h_p $ in terms of the exit measure $p^ \infty$ of the random walk on
the geometric boundary $\partial G$ and the Martin kernel.
Unfortunately, it is not clear in that generality that the measure $p^
\infty$, seen as a linear functional on H\"older continuous functions
on the geometric boundary, depends smoothly on the probability in
$M(F)$. We use a symbolic representation of $\partial G$ (see
\cite{CP}) to express $p^ \infty$ as an combination of a finite
number of symbolic measures. Each of these symbolic measures depends
Lipschitz on $p$, and the entropy $h_p$ is the maximum of a finite
family of Lipschitz functions. The escape rate is expressed by an
analogous formula: it is the maximum of the integrals of the Busemann
kernel with respect of the stationary measures on the Busemann
boundary. It turns out that the Busemann boundary can be described in
terms of the same symbolic representation, and the Lipschitz regularity
of the escape rate follows. It is likely that both entropy and rate of
escape are more regular than what is obtained here, but this is what we
can prove for the moment. Observe that for $G = \Z$, $S=\{ \pm1\}$,
and $F$ a finite generating subset, the function $p \mapsto\ell_p^S =
| {\sum_F i p_i} | $ is Lipschitz continuous on $\M(F)$, but not $C^1$.
For another example in the same spirit, we recall that Mairesse and
Matheus~\cite{MM} have shown that for the braid group $B_3 =
\langle a, b
| aba=bab \rangle$ and $F = \{a, a^{-1}, b, b^{-1} \}$, $p \mapsto
\ell_p^F$ is Lipschitz, but not $C^1$ on $\M(F)$. The entropy is constant
$0$ in the case of $\Z$; the regularity of the entropy for the braid
group is unknown.

In this note, the letter $C$ stands for a real number independent of
the other variables, but which may vary from line to line. The lower
case $c_0, c_1$ will be constants which might depend only on $p \in\M
(F)$. In the same way, the letter $ \O_p$ stands for a neighborhood of
$p $ in $ \M(F)$ which may vary from line to line.

\section{Preliminaries}\label{secpreliminaries}

\subsection{Hyperbolic groups}
We first recall basic facts about hyperbolic\break groups~\cite{GH}. Let
$G$ be a finitely generated group with a symmetric finite set of
generators~$S$. Let $d(x,y) = |x^{-1}y|$ be the word metric on $G$
associated to $S$. For a subset $F\subset G$, we denote
\[
N(F, R):= \bigl\{ x \in G\dvtx d(x,F) \leq R \bigr\}
\quad\mbox{and}\quad
\partial F =
\bigl\{ x \in G\dvtx d(x,F) = 1 \bigr\}.
\]
For $x,y,z \in G$, the \textit{Gromov product} $(x|y)_z$ is defined by
the formula
\[
(x|y )_z = \tfrac{1}{2} \bigl( d(x,z) + d(y,z) - d(x,y)
\bigr).
\]
We write $(x|y)$ for $(x|y)_e$, where $e$ is the unit element. Let
$\delta>0$. The group $G$ is said to be $\delta$-hyperbolic if, for
all $x,y,z,w \in G$,
%
%
\begin{equation}
\label{hyp} (x|y)_w \geq\min\bigl\{ (x|z)_w,
(y|z)_w \bigr\} - \delta.
\end{equation}
If $G$ is $\delta$-hyperbolic, then every geodesic triangle $\D= \{\a
, \b, \gamma\}$ in $G$ is $4\delta$-slim, that is,
\[
\a\subset N(\b\cup\gamma, 4\delta),\qquad
\b\subset N(\gamma\cup\a, 4\delta),\qquad
\gamma\subset N(\a\cup\b, 4\delta).
\]

A sequence $\{x_n \}_{n\geq1}$ is said to converge to infinity if
$\lim_{n,m \to\infty} (x_n|x_m) = \infty$. Two sequences $\{x_n \}
_{n\geq1}$ and $\{y_n \}_{n\geq1}$ converging to infinity are said to
be equivalent if $\lim_{n \to\infty} (x_n|y_n) = \infty$. The
geometric boundary $\partial G$ is defined as the set of equivalence
classes of sequences converging to infinity. The Gromov product extends
to $ G \cup\partial G$ by setting
\[
(\xi|\eta) = \sup\liminf_{n,m \to\infty} (x_n|y_m),
\]
where the sup runs over all sequences $\{x_n \}_{n\geq1}$ converging
to $\xi$ and $\{y_m \}_{m\geq1}$ converging to $\eta$. Recall that
$G\cup\partial G$ is compact when equipped with the base $\{N(\{x\}
,r)\} \cup\{ V_r(\xi) \}, $ where
\[
V_r (\xi):= \bigl\{ \eta\in G \cup\partial G\dvtx (\eta|\xi) > r \bigr
\}.
\]
One can introduce a metric $\rho$ on $\partial G$ such that, for some
$a >1 $ and $C>0$,
\[
a^{-(\xi|\eta) -C} \leq\rho(\xi, \eta) \leq a^{-(\xi
|\eta) +C}.
\]

Another boundary is the Busemann boundary $\partial_BG$. Define, for
$x\in G$, the function $\Psi_{x}(z)$ on $G$ by
\[
\Phi_{x}(z) = d(x,z)-d(x,e).
\]
The assignment $x\mapsto\Psi_{x}$ is continuous, injective and takes
values in a relatively compact set of functions for the topology of
uniform convergence on compact subsets of $G$. The \textit{Busemann
compactification} $\overline{G}$ of $G$ is the closure of $G$
for that topology. The Busemann compactification $\overline G$ is a compact
$G$-space. The \textit{Busemann boundary} $\partial_B G:=\overline
{G}\setminus G$
is made of Lipschitz continuous functions $h$ on $G$ such that
$h(e)=0$, and such that the Lipschitz constant is at most 1. Moreover,
they are horofunctions in the sense of~\cite{CP}: they have the
property that for all $\l\leq h(x)$, the distance of a point $x$ to
the set $h^{-1} (\l) $ is given by $h(x) - \l$; see Section \ref
{Markov} for more about horofunctions.

\subsection{Random walks}

Let $\Xi$ be a compact space. $\Xi$ is called a $G$-space if the
group $G$ acts by continuous transformations on $\Xi$. This action
extends naturally to probability measures on $\Xi$. We say that the
measure $\nu$ on $\Xi$ is stationary if $\sum_{x \in G}( x_\ast\nu
) p(x) = \nu$. The \textit{entropy} of a stationary measure $\nu$ is
defined by
%
%
\begin{equation}
\label{entr} h_p (\Xi,\nu) = - \sum_{x \in G}
\biggl( \int_{\Xi}\ln\frac{dx^{-1}_\ast\nu}{d\nu} (\xi) \,d\nu(\xi)
\biggr) p(x).
\end{equation}
The entropy $h_p$ and the escape rate $\ell_p$ are given by
variational formulas over stationary measures (see~\cite{KV}, Section
3, for the entropy and~\cite{KL}, Theorem 18, for the escape rate)
%
%
\begin{eqnarray}
\label{entropy} h_p &=& \max\bigl\{ h_p (\Xi,\nu); \Xi
G\mbox{-space and } \nu \mbox{ stationary on } \Xi\bigr\},
\\
\label{escape} \ell_p^S & = &\max\biggl\{\sum
_{x \in G} \biggl( \int_{\overline G} h
\bigl(x^{-1}\bigr) \,d\nu(h) \biggr) p(x); \nu\mbox{ stationary on }
\overline G \biggr\}.
\end{eqnarray}
Moreover, the stationary measures in (\ref{escape}) are supported by
$\partial_B G$. In particular, in the case when $G$ is virtually
cyclic, $\partial_B G$ is finite and not reduced to a
point,\setcounter{footnote}{1}\footnote{The restriction
of each limit function to a $\Z$ coset is of the form
$\pm x + a$, where $a$ can take a finite number of values, and there is
at least one $\Z$ coset where both signs appear.} and $\ell_p^S$ is
given by the maximum of a finite number of linear functions of $p$.

Let $\Om= G^\N$ be the space of sequences of elements of $G$, $M$ the
product probability $p^\N$. The random walk is described by the
probability $\P$
on the space of paths~$\Om$, the image of $M$ by the mapping
\[
(\omega_n )_{n \in\Z} \mapsto(X_n
)_{n \geq0 }\qquad\mbox{where } X_0 = e \mbox{ and }
X_n = X_{n-1}\omega_n \mbox{ for } n>0.
\]
In particular, the distribution of $X_n$ is the convolution $p^{(n)}$.
We have:
%
%
\begin{theo}[(\cite{An}, Corollary 6.3,~\cite{K}, Theorem 7.5)]
\label{stat} There is a mapping $X_\infty\dvtx \Om\to\partial G$ such that for
$M$-a.e. $\omega$,
\[
\lim_n X_n (\omega) = X_\infty(\omega).
\]
\end{theo}
The action of $G$ over itself by left multiplications extends to
$\partial G$ and makes $\partial G$ a $G$-space. The image measure
$p^\infty:= (X_\infty)_\ast M$ is the only stationary probability
measure on $\partial G$, and $(\partial G, \nu)$ achieves the maximum
in (\ref{entropy}) (\cite{K}, Theorem 7.6)
%
%
\begin{equation}
\label{entropy2} h_p = h_p \bigl(\partial G,
p^\infty\bigr) = - \sum_{x \in F} \biggl( \int
_{\partial G}\ln\frac{dx^{-1}_\ast p^\infty}{dp^\infty} (\xi) \,dp^\infty
(\xi)
\biggr) p(x).
\end{equation}

The Green function $G(x)$ associated with $(G,p) $ is defined by
\[
G(x) = \sum_{n=0}^\infty p^{(n)}(x)
\]
(see, e.g., Proposition~\ref{decay} for the convergence of the series).
For $y \in G$, the Martin kernel $K_y$ is defined by
\[
K_y (x) = \frac{G(x^{-1}y)}{G(y)}.
\]
Ancona (\cite{An}, Th\'{e}or\`{e}me 6.2) showed that $y_n \to\xi\in\partial
G$ if, and only if, the Martin kernels $K_{y_n} $ converge toward a
function $K_\xi$ called the Martin kernel at $\xi$. We have
%
%
\begin{equation}
\label{density1} \frac{dx_\ast p^\infty}{dp^\infty
}(\xi) = K_\xi(x).
\end{equation}

\subsection{Differentiability}

We are going to use formula (\ref{entropy}) and first show that the
mapping $p \mapsto-\ln K_{\xi}(x)$ is Lipschitz continuous from a
neighborhood $\O_p$ of $p$ in $\M(F)$ into a space of H\"older
continuous functions on $\partial G$.
The following properties are obtained exactly in the same way as in
\cite{L2}.

For $x,y \in G$, let $u(x,y)$ be the probability of eventually reaching
$y$ when starting from $x$. By left invariance, $u(x,y ) = u (e,
x^{-1}y) $. Moreover, by the strong Markov property, $G(x) = u(e,x)
G(e)$ so that we have
%
%
\begin{equation}
\label{Martin} K_y(x) = \frac{u(x,y)}{u(e,y)}.
\end{equation}
By definition, we have $0< u(x,y) \leq1$. The number $u(x,y) $ is
given by the sum of the probabilities of the paths going from $x$ to
$y$ which do not visit $y$ before arriving at $y$. The next two results
are classical:

%
\begin{prop}\label{decay} Let $p \in\M(F)$. There are numbers $C$
and $\zeta, 0 < \zeta< 1$ and a neighborhood $\O_p$ of $p$ in $\M
(F)$ such that for all $q \in\O_p$, all $x \in G$ and all $n \geq0$,
\[
q^{(n)} (x)\leq C \zeta^n.
\]
\end{prop}
\begin{pf} Let $q \in\M(F)$. Consider the convolution operator $P_q$
in $\ell_2 (G, \R)$, defined by
\[
P_qf(x) = \sum_{y\in F} f
\bigl(xy^{-1}\bigr) q(y).
\]
Derriennic and Guivarc'h~\cite{DG} showed that for $p \in\M(F)$,
$P_p$ has spectral radius smaller than one. In particular, there exists
$n_0$ such that the operator norm of $P_p^{n_0} $ in $\ell_2(G)$ is
smaller than one. Since $F$ and $F^{n_0}$ are finite, there is a
neighborhood $\O_p$ of $p$ in $\M(F)$ such that for all $q \in\O_p$,
$\| P_q^{n_0} \|_2 < \l$ for some $\l<1$ and $\| P_q^k \|_2 \leq
C$ for $1\leq k \leq n_0$.
It follows that for all $q \in\O_p $, all $n \geq0$,
\[
\bigl\|P_q^n \bigr\|_2 \leq C\l^{[n/n_0]}.
\]
In particular, for all $x \in G$, $q^{(n)} (x) = [P_q^n \delta_e](x)\leq
|P_q^n \delta_e|_2 \leq C \l^{[n/n_0]} | \delta_e |_2
\leq C\l^{[n/n_0]}$.
\end{pf}
%
%
\begin{cor}[(\cite{DG})]\label{decayG} Let $p \in\M(F)$. There are
numbers $C$ and $\delta>0$ such that for all $q \in\O_p$, all $x,y
\in G$,
\[
G(x,y) \leq C e^{-\delta|x^{-1}y|}.
\]
\end{cor}
\begin{pf}
We have $q^{(n)} (x^{-1}y ) = 0 $ for $ n \leq\frac{1}{r} |x^{-1}y|
$; take $\delta= \frac{1}{r} \ln\frac{1}{\zeta} $.
\end{pf}

Fix $p \in\M(F)$, and let $\D$ be a subset of $G$. We can define
$G_\D(x,y), u_\D(x,y) $ by considering only the paths of the random
walk which remain inside $\D$. Clearly, $G_\D\leq G, u_\D\leq u$.
For $x\in G$, $V$ a subset of $G$ and $v\in V$, let $\a_x^V(v)$ be the
probability that the first visit in $V$ of the random walk starting
from $x$ occurs at $v$ $ (\a_x^V (v) = u_{G\setminus V \cup\{v\}}
(x,v)$). We have $0 \leq\sum_{v\in V} \a_x^V (v) \leq1$ and the following:
%
%
\begin{prop}\label{cont} Fix $x$ and $V$. For all $s >1$, the mapping
$p \mapsto\a_x^V(v) $ is a $C^\infty$ function from $\M(F)$ into
$\ell^s (V)$. Moreover, $\|\frac{\partial\a_x^V}{\partial p_i}\|_s
$ is bounded independently of $x$ and $V$.
\end{prop}
\begin{pf} By Proposition~\ref{decay}, there is a neighborhood $\O_p$
of $p$ in $\M(F)$ and numbers $C$, $\zeta, 0 < \zeta< 1$, such that
for $q \in\O_p$ and for all $y \in G$,
\[
q^{(n)} (y) \leq C \zeta^n.
\]
The number $\a_x^V(v)$ can be written as the sum of the probabilities
$\a_x^{n,V} (v) $ of entering $V$ at $v$ in exactly $n$ steps. We have
\[
\a_x^{n,V} (v) \leq q^{(n)} \bigl(x^{-1}v
\bigr) \leq C \zeta^n.
\]
Moreover, the function $p \mapsto\a_x^{n,V} (v) $ is a homogeneous
polynomial of degree $n$ on~$\M(F)$, since
\[
\a_x^{n,V} (v) = \sum_\EE
q_{i_1}q_{i_2}\cdots q_{i_n},
\]
where $\EE$ is the set of paths $\{ x, xi_1, xi_1i_2, \ldots, xi_1i_2
\cdots i_n = v \}$ of length $n$ made of steps in $F$ which start from
$x$ and enter $V$ in $v$. It follows that for all $\alpha= \{ n_1,
n_2, \ldots, n_{|B|}, n_i \in\N\cup\{0\} \}$, all $v\in V$,
\[
\biggl|\frac{\partial^\alpha}{\partial p^\alpha} \a_x^{n,V} (v) \biggr| \leq
\frac{ n^{|\alpha|}}{(\inf_{i \in F} p_i)^{|\alpha|}} \a_x^{n,V} (v)
\leq\frac{C n^{|\alpha|}}{(\inf_{i \in F}
p_i)^{|\alpha|}}
\zeta^n,
\]
where $|\a| = \sum_{i \in F} n_i$. Therefore,
\[
\sum_{v \in V} \biggl|\frac{\partial^\alpha}{\partial p^\alpha}
\a_x^{n,V} (v) \biggr|^s \leq\frac{C n^{s|\alpha|}}{(\inf_{i \in F}
p_i)^{s|\alpha|}}
\zeta^{(s-1)n} \sum_{v \in V} \a_x^{n,V}
(v) \leq\frac{C n^{s|\alpha|}}{(\inf_{i \in F} p_i)^{s|\alpha|}} \zeta
^{(s-1)n}.
\]
Thus, $q \mapsto\frac{\partial^\alpha}{\partial p^\alpha} \a_x^V(v)$ is
given locally by a uniformly converging series in $\ell^s
(V)$ of derivatives. It follows that $q \mapsto\a_x^V(v)$ is a
$C^\infty$ function from $\M(F)$ into $\ell^s(V)$. From the above
computation, it follows that $\| \frac{\partial^\alpha}{\partial
p^\alpha} \a_x^V(v) \|_s \leq\sum_n \frac{C n^{s|\alpha|}}{(\inf_{i \in
F} p_i)^{s|\alpha|}} \*\zeta^{(s-1)n}$, independently of $x,V$.
\end{pf}

%
\begin{prop}\label{cont*} There exists $T$ large enough that for $t
>T$, for any $y$ and~$V$, the mapping\vspace*{1pt} $p \mapsto\a_v^{\{y\}}(y) $ is
a $C^\infty$ function from $\M(F)$ into $\ell^t (V)$. Moreover, $v
\mapsto\frac{\partial\a_v^{\{y\}}(y) }{\partial p_i} $ is bounded
in $\ell^t(V)$ independently of $y$ and $V$.
\end{prop}
\begin{pf} It suffices to show that there is $T$ such that the function
$v \mapsto\a_v^{\{y\}}(y) \in\ell^T (V)$ and to apply the same
arguments as in the proof of Proposition~\ref{cont}. Consider the
probability $\check{p} $ with support\vspace*{1pt} $F^{-1}$ defined by
$\check{p} (x) = p(x^{-1})$, and define all quantities $(\check
{p})^{(n)}, \check{G} (x), \check{u}(x,y)$. Observe that, since it is
the sum of the same probabilities over the same set of paths, $ G (v,y)
= \check G(y,v)$. Therefore, we have, using Corollary \ref {decayG} for
the $\check p$ random walk,
\[
\a_v^{\{y\}}(y) \leq G(v,y) = \check G(y,v) \leq
e^{-\check\delta|y^{-1}v|}.
\]
The group $G$ has exponential growth: there is a $v$ such that there
are less than $Ce^{vR}$ elements of $G$ at distance less than $R$ from
$y$. It follows that for $T > v/\check\delta$, the function $v
\mapsto\a_v^{\{y\}}(y) \in\ell^T (V)$.
\end{pf}

\subsection{Projective contractions on cones}

In this subsection, we recall the Birkhoff theorem about linear maps
preserving convex cones. Let $\CC$ be a convex cone in a Banach space,
and define on $\CC$ the projective distance between half lines as
\[
\vt(f,g):= \ln\bigl[\tau(f,g) \tau(g,f)\bigr],
\]
where $\tau(f,g):= \inf\{ s, s>0, sf-g \in\CC\}$. Let $\DD$ be
the space of directions in~$\CC$. Then, $\vt$ defines a distance on
$\DD$. Let $A$ be an operator from $\CC$ into $\CC$, and let $T\dvtx
\DD\to\DD$ be the projective action of $A$. Then, by~\cite{B},
%
%
\begin{equation}
\label{contr} \vt(Tf,Tg) \leq\b\vt(f,g)\qquad \mbox{where } \b= \tanh
\bigl(\tfrac{1}{4} \operatorname{Diam}T(\DD) \bigr).
\end{equation}

In some cases, $\vt$-diameters are easy to estimate: for example, in
$\CC^t= \{ f \in\ell^t; f \geq0 \} $, the set $ \mathcal U (g,c):=
\{ f\dvtx c^{-1} g \leq f \leq cg \}$, where $g\in\CC$ and $c \geq1$,
has $\vt$-diameter $4 \ln c$. Moreover, the following observation is useful:
%
%
\begin{lem}[(\cite{Li}, Lemma 1.3)]\label{Liverani} Let $f,g \in\CC^t, \|
f \|_t = \|g\|_t $. Then,
\[
\|f-g\|_t \leq\bigl(e^{\vt(f,g)} -1 \bigr) \|f\|_t.
\]
\end{lem}
%

\section{Obstacles}\label{secMartin}

In this section, we show that the function $\Phi$ on $\partial G$
defined by $\Phi(\xi):= -\ln K_{\xi}(x)$ is H\"older continuous for
any fixed $x \in G$. This is not a new result~\cite{INO}.
Nevertheless, we present the construction and the proof in order to
introduce the notation used in the next section to show that $\Phi$
is Lipschitz in $p$ as a H\"older continuous function on $\partial G$.
Like in~\cite{INO}, the proof is based on Ancona's Harnack inequality
at infinity (see~\cite{An} and~\cite{INO}, Proposition 2.1, for the form
used here): there exist a number $R$ and a constant $c= c(p)$ such that
if $[x,y]$ is a geodesic segment and $z \in[x,y]$, then for any $\D
\subset G$, $N([x,y], R) \subset\D$, we have
%
%
\begin{equation}
\label{Harnack} c^{-1} u_\D(x,z) u_\D(z,y)
\leq u_\D(x,y) \leq c u_\D(x,z) u_\D(z,y),
\end{equation}
where $u_\D(v,w)$ is the probability of ever arriving at $w$ starting
from $v$ before reaching $G\setminus\D$. Moreover, from the proof of
(\ref{Harnack}) in~\cite{INO} or~\cite{W}, it follows that there
exists a neighborhood $\O$ of $p$ in $\M(F)$ and a constant $C$ such
that $c(p) \leq C$ for $p \in\O$.

\subsection{Obstacles}

Without loss of generality, we may assume that $F$ contains the set of
generators, and $\delta$ is an integer. Set $r = \max\{ |x|; x \in F,
\delta\}$.

Fix $M$ large. In particular, $M \geq R + 12 r $, where $R$ is given by
(\ref{Harnack}). For a geodesic~$\gamma$, we call an \textit{obstacle} a
family $U_0^-\subset U_0 \subset U_1^- \subset U_1$ of subsets of $G$
such that
\begin{eqnarray*}
U_0^- &=& \bigl\{ x \in G\dvtx d\bigl(x, \gamma(-2M) \bigr) < d\bigl(x,
\gamma(0) \bigr)\bigr\},
\\
U_0 & =& \bigl\{ x \in G\dvtx d\bigl(x, \gamma(-2M) \bigr) < d\bigl(x,
\gamma(4r)\bigr) \bigr\},
\\
U_1^- & = & \bigl\{ x \in G\dvtx d\bigl(x, \gamma(0) \bigr) < d\bigl(x,
\gamma(2M)\bigr) \bigr\},
\\
U_1 &= & \bigl\{ x \in G\dvtx d\bigl(x, \gamma(0) \bigr) < d\bigl(x,
\gamma(2M+ 4r)\bigr)\bigr\}.
\end{eqnarray*}
The subsets $U_i^\pm$ are connected and satisfy
$U_0^-\subset U_0 \subset U_1^- \subset U_1$. More precisely, we have
the two following elementary facts:
%
%
\begin{lem}\label{connexe} If $x \in U_0^- $ and $[x,\gamma(-2M)] $
is a geodesic segment, then $[x,\gamma(-2M)] \subset U_0^-$.
\end{lem}
\begin{pf} Assume not. Then there is a $z \in[x,\gamma(-2M)] $ such
that\break $ d(z, \gamma(-2M)) \geq d(z,\gamma(0))$. Adding $d(z,x) $ to
both sides of this inequality, we obtain
\[
d\bigl(x, \gamma(-2M)\bigr)= d(x,z) + d\bigl(z, \gamma(-2M)\bigr) \geq
d(x,z ) +
d\bigl(z,\gamma(0)\bigr) \geq d\bigl(x, \gamma(0)\bigr),
\]
a contradiction to $x \in U_0^-$.
\end{pf}
The statements and the proofs are the same for all $U_i^\pm$.
%
%
\begin{lem}\label{geomofobstacles}
If $x \in U_0^-$, then $B(x,r) \subset U_0$; if $x \in U_0$, then $B(x,
M -3r ) \subset U_1^-$.
\end{lem}
\begin{pf}
Let $x \in U_0^-$ and $x' \in B(x,r)$. Writing (\ref{hyp}) with $x =
x', y= \gamma(0), z= \gamma(4r) $ and $w = \gamma(-2M) $, we get
\begin{eqnarray*}
&&
d\bigl(x',\gamma(-2M)\bigr) -d\bigl(x',\gamma(0)
\bigr) + 2M \\
&&\qquad\geq\min\bigl\{ d\bigl(x',\gamma(-2M)\bigr) -d
\bigl(x',\gamma(4r)\bigr) + 2M +4r, 4M \bigr\} -2\delta.
\end{eqnarray*}
Since $ d(x',\gamma(-2M)) -d(x',\gamma(4r)) < d(x,\gamma(-2M))+r
-d(x,\gamma(0)) +5r < 6r \leq2M -4r$, we get
\begin{eqnarray*}
d\bigl(x', \gamma(4r)\bigr) &\geq& d\bigl(x',
\gamma(0)\bigr) +4r - 2\delta > d\bigl(x, \gamma(0)\bigr) +r \\
&>& d\bigl(x,
\gamma(-2M)\bigr) + r > d\bigl(x', \gamma(-2M)\bigr).
\end{eqnarray*}
Analogously, if $x \in U_0 $ and $x' \in B(x, M-3r)$, we get, writing
now (\ref{hyp}) with $ z= \gamma(2M) $,
\begin{eqnarray*}
&&
d\bigl(x',\gamma(-2M)\bigr) -d\bigl(x',\gamma(0)
\bigr) + 2M \\
&&\qquad\geq\min\bigl\{ d\bigl(x',\gamma(-2M)\bigr) -d
\bigl(x',\gamma(2M)\bigr) + 4M, 4M \bigr\} -2\delta.
\end{eqnarray*}
Since the right-hand side is smaller than $4M -2r$, it cannot exceed
$4M - 2\delta$, and we get
\[
d\bigl(x', \gamma(0) \bigr) \leq d\bigl(x',
\gamma(2M) \bigr) -2M + 2\delta< d\bigl(x', \gamma(2M)\bigr).
\]
\upqed\end{pf}

Lemma~\ref{geomofobstacles} implies that
any trajectory of the random walk going from $U_0^-$ to $G \setminus
U_1$ has to cross successively $U_0 \setminus U_0^-, U_1^-\setminus
U_0$ and $U_1 \setminus U_1^-$. For $V_1,V_2$ subsets of $G$, denote
$A_{V_1}^{V_2}$ the (infinite) matrix such that the row vectors indexed
by $v \in V_1$ are the $\a_v^{V_2} (w), w \in V_2$. In particular, if
$V_2 = \{y\}$, set $\omega_{V_1}^y $ for the (column) vector
\[
\omega_{V_1}^y = A_{V_1}^{\{y\}} =
\bigl(\a_v^{\{y\}} (y)\bigr)_{v\in V_1} = \bigl(u(v,y)
\bigr)_{v\in V_1}.
\]
Fix $t >T$. By Propositions~\ref{cont} and~\ref{cont*}, $\omega
_{V_1}^y$ is a vector in $\ell^t (V_1)$ and $\a_x^{V_0} \in\ell
^s(V_0)$, with $1/s +1/t = 1$.
With this notation, the strong Markov property yields, if $U_0^-\subset
U_0 \subset U_1^- \subset U_1$ is an obstacle and $x \in U_0^-$, $y
\notin U_1$,
\[
u(x,y) = \sum_{v_0,v_1} \a_x^{V_0}
(v_0) A_{V_0}^{V_1} (v_0,v_1)
u(v_1,y) = \bigl\langle\a_x^{V_0},A_{V_0}^{V_1}
\omega_{V_1}^y \bigr\rangle
\]
with the natural summation rules for matrices and for the ($\ell^s,
\ell^t$) coupling. All series are bounded series with nonnegative
terms, and we set \mbox{$V_i =
U_i\setminus U_i^-$}.

Observe that an obstacle is completely determined by the directing
geodesic segment $[\gamma(-2M), \ldots, \gamma(2M+4r)]$, so that
there is a finite number of possible obstacles and therefore a finite
number of spaces $\ell^t(V)$, of (infinite) matrices $A_{V_0}^{V_1}$,
of vectors $\omega_{V_1}^{z}$ and $\a_x^{V_0}$ if the distances
$d(z,\gamma(2M+4r +1))$ and $d(x,\gamma(-2M))$ are bounded.

\subsection{Properties of the matrix $A_{V_0}^{V_1}$}

Recall that the general entry of the matrix $A = A_{V_0}^{V_1}$ is
$A(v_0,v_1) $, the probability that starting from $v_0 \in V_0$, the
first visit in $V_1$ occurs at $v_1$. In particular, assume $A(v_0,v_1)
= 0$. Then, all paths from $v_0$ to $v_1$ with steps in $F$ have to
enter $V_1$ elsewhere before reaching $v_1$. Since the support $F$ of
$p$ contains the generators of the group, $A(v,v_1) = 0$ for all $v$'s
in the connected component of $v_0$ in $U_1^-$. By Lemma~\ref
{connexe}, all paths from $\gamma(0)$ to $v_1$ with steps in $F$ have
to enter $V_1$ before reaching $v_1$. Therefore this property depends
neither on $v_0 \notin U_1^-$ nor on $p \in\M(F)$. We say that
$v_1$ is active if $A(v_0,v_1) \not= 0 $. In the sequel we will call
$V_1$ the set of active elements of $U_1 \setminus U_1^-$.
We have:
%
%
\begin{prop}
Let $\gamma$ be a geodesic, $U_0^-\subset U_0 \subset U_1^- \subset
U_1$ an obstacle, $ V_0 = U_0 \setminus U_0^-$, $V_1$ the active part
of $U_1 \setminus U_1^-$. There exists a neighborhood $\O_p$ of $p $
in $\M(F)$ and a constant $c_1$ such that, for all $p \in\O_p$, all
$v_0 \in V_0, v_1 \in V_1$,
%
%
\begin{equation}
\label{product} c_1^{-1} u_{G\setminus U_1^-}
\bigl(v_0, \gamma(0)\bigr) \a_{\gamma(0)}^{V_1} (
v_1) \leq A (v_0, v_1) \leq c_1
u_{G\setminus U_1^- } \bigl(v_0, \gamma(0)\bigr) \a_{\gamma(0)}^{V_1}
( v_1).\hspace*{-25pt}
\end{equation}
\end{prop}
\begin{pf} Introduce the set $U_1^{--}$, $U_1^{--} =\{ x \in G\dvtx d(x,
\gamma(0) ) < d(x,\gamma(2M -4r)) \}$. By a variant of Lemma \ref
{geomofobstacles}, we may write, for $v_0 \in V_0, v_1 \in V_1$,
\[
A(v_0,v_1) = \sum_{w \in U_1^- \setminus U_1^{--}}
u_{G\setminus
U_1^-} (v_0,w) \a_w^{V_1}
(v_1).
\]
Using that $\a_{\gamma(0)}^{V_1} (v_1) = \sum_{w \in U_1^- \setminus
U_1^{--}} u_{G\setminus U_1^-} (\gamma(0), w) \a_w^{V_1} (v_1) $, we
see that it suffices to prove that, for all $p \in\O_p$, all $v_0 \in
V_0, w \in U_1^{--} \setminus U_1^-$,
\begin{eqnarray*}
&&
c_1^{-1} u_{G\setminus U_1^-} \bigl(v_0,
\gamma(0)\bigr) u_{G\setminus U_1^-} \bigl(\gamma(0), w\bigr) \\
&&\qquad\leq
u_{G\setminus U_1^-}
(v_0, w) \leq c_1 u_{G\setminus U_1^-}
\bigl(v_0, \gamma(0)\bigr) u_{G\setminus U_1^-} \bigl(\gamma(0), w
\bigr).
\end{eqnarray*}
This will follow from a variant of (\ref{Harnack}) once we will have
located the point $\gamma(0)$ with respect to the geodesic $[v_0,w]$.

Observe that if $v_0 \in U_0$, then $d(v_0, \gamma(0) ) \geq M-3r$.
Indeed, writing that
\begin{eqnarray*}
&&
\bigl(\gamma(-2M), \gamma(4r) \bigr)_{v_0} \\
&&\qquad \geq\min\bigl\{ \bigl(
\gamma(-2M), \gamma(-M+2r)\bigr)_{v_0}, \bigl(\gamma(4r), \gamma(-M+2r)
\bigr)_{v_0} \bigr\} -\delta
\\
&&\qquad = \bigl(\gamma(-2M), \gamma(-M+2r)\bigr)_{v_0} - \delta,
\end{eqnarray*}
we get that $d(v_0, \gamma(4r) )\geq M+2r - \delta\geq M-3r $ and the
claim follows. Since, by Lemma~\ref{connexe}, the whole geodesic
$[v_0, \gamma(-M +2r)]$ lies in $U_0$, we have $d( \gamma(0),
[v_0,\gamma(-M+2r)]) \geq M-3r$. But we know that $ \gamma(0) \in N(
[v_0,\gamma(M)] \cup[v_0, \gamma(-M+2r)], 4\delta)$. It follows
that there is a point $z_1 \in[v_0,\gamma(M)]$ with $d(\gamma
(0),z_1) \leq4\delta$.
In the same way, since $w \in G \setminus U_1^{--}$, $d( \gamma
(0),[w,\gamma(M-2r)] ) \geq M-3r $ and therefore $d(z_1,[w,\gamma
(M-2r)] ) \geq M-3r - 4\delta$. It follows that there is a point $z
\in[v_0,w]$ such that $d(z, \gamma(0)) \leq d(z,z_1 ) + d(z_1,
\gamma(0)) \leq8\delta$.

Let $y_0$ be the point in $[v_0,w]$ at distance $R$ from $w$. Then
$G\setminus U_1^- $ contains $N([v_0,y_0], R)$, and the point $z$
belongs to $[v_0,y_0]$.\footnote{Since $w \notin U_1^{--}$, we have
$d(w,z) \geq d(w, \gamma(0)) -8\delta\geq M- 3r -8\delta\geq M - 11r
> R$.} So we may apply (\ref{Harnack}) to the points $v_0,z,y_0$ and
the domain $\D= G\setminus U_1^-$ to obtain, for all $p \in\O_p$,
all $v_0 \in V_0, v_1 \in V_1$,
\[
c_0^{-1} u_{G\setminus U_1^-} (v_0, z)
u_{G\setminus U_1^-} (z, y_0) \leq u_{G\setminus U_1^-} (v_0,
y_0) \leq c_0 u_{G\setminus U_1^-} (v_0, z)
u_{G\setminus U_1^-} (z, y_0).
\]
By changing the constant, we can replace $y_0$ by $w$ [since $ d(w,y_0
) = R $] and $z$ by $\gamma(0)$ [since $d(z, \gamma(0)) \leq8 \delta
$]. We obtain the desired inequality.
\end{pf}

For $V$ a subset of $G$, $t>0$, denote $\CC_V^t$ the convex cone of
nonnegative sequences in $\ell^t (V)$ and define on $\CC_V^t$ the
projective distance between half lines as
\[
\vt(f,g):= \ln\bigl[\tau(f,g) \tau(g,f)\bigr],
\]
where $\tau(f,g):= \inf\{ s, s>0, sf-g \in\CC_V^t \}$. Represent
the space of directions as the sector of the unit sphere $ \DD_V^t =
\CC_V^t\cap S_V^t $; then, $\vt$ defines a distance on $\DD_V^t$ for
which $\DD_V^t$ is a complete space (Lemma~\ref{Liverani}). We fix
$t>T $ such that the sequences $ \a_v^{\{y\}}(y) \in\ell^t (V)$ and
we consider the matrix $A_{V_0}^{V_1}$ as an operator from $\ell^t
(V_1)$ into the space of sequences indexed on $V_0$. We have:
%
%
\begin{prop}\label{Acones} Choose $t >T$ and $s$ such that $1/s +1/t =
1$. For any obstacle $U_0^-\subset U_0 \subset U_1^- \subset U_1$, all
$p \in\O_p$, the operator $A_{V_0}^{V_1}$ sends $\CC_{V_1}^t$ into
$\CC_{V_0}^t$, the adjoint operator $(A_{V_0}^{V_1})^\ast$ sends $\CC
_{V_0}^s $ into $\CC_{V_1}^s$ and
\[
\operatorname{Diam}_{\CC_{V_0}^t} \bigl( A_{V_0}^{V_1} \bigl(
\CC_{V_1}^t\bigr) \bigr) \leq4 \ln c_1,\qquad
\operatorname{Diam}_{\CC_{V_1}^s} \bigl( \bigl(A_{V_0}^{V_1}
\bigr)^\ast\bigl(\CC_{V_0}^s\bigr) \bigr) \leq4
\ln c_1,
\]
where $c_1$ and $\O_p$ are the ones in (\ref{product}).
\end{prop}
\begin{pf} By definition, $ \sum_{v_1 \in V_1} \a_{\gamma(0)}^{V_1}
(v_1) \leq1$ so that $ \a_{\gamma(0)}^{V_1} (v_1) \in\ell^s
(V_1)$. By~(\ref{product}), for any $v_0 \in V_0$, any $f \in\ell^t
(V_1)$,
\[
A_{V_0}^{V_1}f (v_0) \leq c_1
u_{G\setminus U_1} \bigl(v_0, \gamma(0)\bigr) \bigl\|
\a_{\gamma(0)}^{V_1} (v_1) \bigr\|_s \|f
\|_t.
\]
By the same argument as in the proof of Proposition~\ref{cont*}, we
see that $v_0 \mapsto u_{G\setminus U_1^-} (v_0, \gamma(0))
\in\ell^t(V_0)$.
It follows that for any $f \in\ell^t (V_1)$, $A_{V_0}^{V_1} f $
belongs to $ \ell^t(V_0)$.

By (\ref{product}), we know that for any $f \in\CC_{V_1}^t$,
%
%
\begin{equation}
\label{cone} c_1^{-1} u_{G\setminus U_1}
\bigl(v_0, \gamma(0)\bigr) \leq\frac{A_{V_0}^{V_1} f (v_0)} {\langle \a
_{\gamma(0)}^{V_1}
(\cdot),f(\cdot)\rangle} \leq
c_1 u_{G\setminus U_1} \bigl(v_0, \gamma(0)\bigr).
\end{equation}
It follows that $\operatorname{Diam}_{\CC_{V_0^t}}A_{V_0}^{V_1} (\CC
_{V_1}^t) \leq4 \ln c_1$.
The same argument works for the adjoint operator $(A_{V_0}^{V_1})^\ast
$, since we know that $v_0 \mapsto u_{G\setminus U_1} (v_0, \gamma
(0)) \in\ell^t (V_0)$ and $v_1 \mapsto\a_{\gamma(0)}^{V_1} (v_1)
\in\ell^s(V_1)$.
\end{pf}

%
\begin{prop}\label{Acont} Choose $t >T + 1, s$ such that $1/s + 1/t =
1$. The mapping $ p \mapsto A_{V_0}^{V_1} $ [resp., $p \mapsto
(A_{V_0}^{V_1})^\ast$] is $C^\infty$ from $\M(F)$ into $\mathcal L
(\ell^t (V_1), \ell^t(V_0))$ [resp., $\mathcal L (\ell^s (V_0),
\ell^s(V_1))$].
\end{prop}
\begin{pf} We follow the scheme of the proofs of Propositions \ref
{cont} and~\ref{cont*}. By Proposition~\ref{decay}, there is a
neighborhood $\O_p$ of $p$ in $\M(F)$ and numbers $C$, $\zeta, 0 <
\zeta< 1$, such that for $q \in\O_p$ and for all $y \in G$,
\[
q^{(n)} (y) \leq C \zeta^n.
\]
We write $\a_{v_0}^{V_1}(v_1)$ as the sum of the probabilities $\a
_{v_0}^{n,V_1} (v_1) $ of entering $V_1$ at $v_1$ in exactly $n$
steps. We have, for all $v_0 \in V_0, v_1 \in V_1$,
\[
\a_{v_0}^{n,V_1} (v_1) \leq q^{(n)}
\bigl(v_0^{-1}v_1\bigr) \leq C
\zeta^n.
\]
As before, the function $p \mapsto\a_{v_0}^{n,V_1} (v_1) $ is a
homogeneous polynomial of degree $n$ on $\M(F)$ and for all $\alpha=
\{ n_1, n_2, \ldots, n_{|B|}, n_i \in\N\cup\{0\} \}$, all $v_0\in
V_0,v_1 \in V_1$,
\[
\biggl|\frac{\partial^\alpha}{\partial p^\alpha} \a_{v_0}^{n,V_1} (v_1) \biggr| \leq
\frac{C n^{|\alpha|}}{(\inf_{i \in F} p_i)^{|\alpha|}} \a_{v_0}^{n,V_1}
(v_1) \leq
\frac{C n^{|\alpha|}}{(\inf_{i
\in F} p_i)^{|\alpha|}} \zeta^n.
\]
Let $f \in\ell^t (V_1)$. Then,
\begin{eqnarray*}
\sum_{v_1} \biggl|\frac{\partial^\alpha}{\partial
p^\alpha}
\a_{v_0}^{n,V_1} (v_1) \biggr| \bigl|f(v_1)\bigr| &\leq&
\biggl\|\frac{C
n^{|\alpha|}}{(\inf_{i \in F} p_i)^{|\alpha|}} \a_{v_0}^{n,V_1} (v_1)
\biggr\|_s \|f\|_t
\\
& \leq & C n^{|\a|} \zeta^{
({s-1})n/{s}}\bigl(u_{G\setminus U_1^-}
\bigl(v_0, \gamma(0)\bigr)\bigr)^{{1}/{s}} \|f\|_t.
\end{eqnarray*}
To obtain the last inequality, we use that
\[
\bigl\| \a_{v_0}^{n,V_1} (v_1)
\bigr\|_s \leq C \biggl( \zeta^{(s-1)n} \sum_{v_1} \a_{v_0}^{n,V_1} (v_1)
\biggr)^{1/s},
\]
(\ref{product}) and $\sum_{v_1} \a_{\gamma
(0)}^{V_1} (v_1) \leq1$.
Therefore,
\begin{eqnarray*}
&&\biggl\| \sum_{v_1} \biggl|\frac{\partial^\alpha}{\partial p^\alpha}
\a_{v_0}^{n,V_1} (v_1) \biggr| \bigl|f(v_1)\bigr|
\biggr\|_{\ell^t(V_0)} \\
&&\qquad\leq C n^{|\a|} \zeta^{({s-1})n/{s}} \biggl(\sum
_{v_0} \bigl(u_{G\setminus U_1} \bigl(v_0,
\gamma(0)\bigr)\bigr)^{{t}/{s}} \biggr)^{{1}/{t}} \|f\|_t.
\end{eqnarray*}
Since $t> T+1$, $t/s >T$, the series $\sum_{v_0} (u_{G\setminus U_1}
(v_0, \gamma(0)))^{{t}/{s}}$ converges, and the operator
\[
f \mapsto\sum_{v_1} \frac{\partial^\alpha}{\partial p^\alpha}
\a_{v_0}^{n,V_1} (v_1) f(v_1)
\]
has norm smaller than $ C n^{|\a|} \zeta^{({s-1})n/{s}}$ in $\LL
(\ell^t(V_1), \ell^t (V_0))$. The series of operators which defines
$ \frac{\partial^\alpha}{\partial p^\alpha}
A_{V_0}^{V_1}$ is converging.

The proof is the same for the adjoint operator. We estimate, for $g \in
\ell^s(V_0)$,
\begin{eqnarray*}
\sum_{v_0} \biggl|\frac{\partial^\alpha}{\partial
p^\alpha}
\a_{v_0}^{n,V_1} (v_1) \biggr| \bigl|g(v_0)\bigr| & \leq &
\biggl\|\frac{C
n^{|\alpha|}}{(\inf_{i \in F} p_i)^{|\alpha|}} \a_{v_0}^{n,V_1} (v_1)
\biggr\|_t \|g\|_s
\\
& \leq & Cn^{|\a|} \zeta^{{n}/{t}} \bigl( \a_{\gamma(0)}^{V_1}
(v_1) \bigr)^{({t-1})/{t}} \|g\|_s.
\end{eqnarray*}
As before, we find that the operator
\[
g \mapsto\sum_{v_0} \frac{\partial^\alpha}{\partial p^\alpha}
\a_{v_0}^{n,V_1} (v_1) g(v_0)
\]
has norm smaller than
\[
C n^{|\a|} \zeta^{{n}/{t}}\sum_{v_1}
\bigl( \a_{\gamma
(0)}^{V_1} (v_1) \bigr)^{s({t-1})/{t}}
\leq C n^{|\a|} \zeta^{{n}/{t}}
\]
in $\LL(\ell^s(V_0), \ell^s (V_1))$ (recall that $s\frac{t-1}{t} =
1$). The series of operators which defines
$ \frac{\partial^\alpha}{\partial p^\alpha}
(A_{V_0}^{V_1})^\ast$ is converging as well.
\end{pf}

\subsection{H\"older regularity of the Martin kernel}\label{Martinreg}

Fix $x \in G$ and a geodesic $\gamma$ with $\gamma(0) = e$. Consider
the family $U_0^-\subset U_0 \subset\cdots\subset U_n^- \subset U_n$
such that for all $j = 1, \ldots, n-1$,
$U_j^-\subset U_j \subset U_{j+1}^- \subset U_{j+1}$ is an obstacle for
$\gamma\circ\s^{2jM+K}$. The integer $K$ is chosen so that $x,e \in
U_0^-$, for example, $K = 4M + |x|$. With that choice, $\gamma(n)
\notin U_k$ as soon as $n > K +2kM +4r$.
Iterating the strong Markov property, we get, for $z \notin U_k$,
\[
\frac{u(x,z)}{u(e, z)} = \frac{\langle\a_x^{V_0},A_{V_0}^{V_1}\cdots
A_{V_{k-1}}^{V_k} \omega_{V_k}^z \rangle
} {\langle\a_e^{V_0},A_{V_0}^{V_1}\cdots A_{V_{k-1}}^{V_k} \omega
_{V_k}^z \rangle}.
\]
Choose $t>T+1$ and $s$ such that $1/s +1/t = 1$. Set $ f_k (z ):=
\frac{\omega_{V_k}^z}{\| \omega_{V_k}^{z}\|_t}, \a:= \a _{e}^{V_0},
\b:= \a_{x}^{V_0}$. For all $ z \notin U_k $, $f_k ( z )
\in\DD_{V_k}^t$ and $\a, \b\in\CC_{V_0}^s -\{0\} $. By Proposition
\ref{cont}, if $z,z' \notin U_k$, $\vt_{\CC^t}
(A_{V_{k-1}}^{V_k}f_k(z),A_{V_{k-1}}^{V_k} f_k(z') ) \leq4\ln c_1$.
Set $\tau= \frac{c_1^2 -1}{c_1^2 +1}$. By repeated
application of (\ref{contr}), we have, as soon as $z,z' \notin U_k$,
%
%
\begin{eqnarray}
\label{exp}
&&
\vt_{\CC^t} \bigl(A_{V_0}^{V_1}\cdots
A_{V_{k-1}}^{V_k} f_k(z), A_{V_0}^{V_1}
\cdots A_{V_{k-1}}^{V_k} f_k \bigl(z'
\bigr) \bigr)\nonumber\\
&&\qquad \leq\tau^{k-1}\vt_{\CC^t} \bigl
(A_{V_{k-1}}^{V_k}f_k(z),A_{V_{k-1}}^{V_k}
f_k\bigl(z'\bigr) \bigr) \\
&&\qquad\leq4\ln c_1
\tau^{k-1}.\nonumber
\end{eqnarray}

We are interested in the function $\Phi\dvtx \partial G \to\R$,
\[
\Phi(\xi) = -\ln K_x (\xi) = -\ln\lim_{x_n \to\xi}
\frac
{u(x,x_n)}{u(e,x_n)}.
\]
If we choose the reference geodesic $\gamma$ converging toward $\xi$,
then setting $\Phi_n (\xi) = \frac{u(x,\gamma(n))}{u(e,\gamma
(n))}$, we have $\lim_n \Phi_n (\xi) = \Phi(\xi)$. More precisely,
as soon as $n,m > K +2kM +4r$, we may write
\[
\Phi_n(\xi) - \Phi_m (\xi) = \ln\frac{ \langle\a,
T_{j_0}\cdots T_{j_{k-1} } f_k(\gamma(n)) \rangle} { \langle\b,
T_{j_0}\cdots T_{j_{k-1}} f_k (\gamma(n)) \rangle}
\frac{ \langle
\b, T_{j_0}\cdots T_{j_{k-1} } f_k(\gamma(m)) \rangle} { \langle
\a, T_{j_0}\cdots T_{j_{k-1}} f_k( \gamma(m)) \rangle},
\]
where $T_{j_s}$ is the projective action of $A_{V_s}^{V_{s+1}}$. By
(\ref{exp}) and Lemma~\ref{Liverani}, we have, as soon as $n,m > K
+2kM +4r$, $|\Phi_n(\xi) - \Phi_m (\xi) |\leq C \tau^k$. For the
same reason, for any fixed family of $f_k \in\DD^t_{V_k} $, the
sequence $T_{j_0}\cdots T_{j_{k-1} } f_k $ converge in $\DD^t_{V_0}$
toward some $f_\infty$, independent of the choice of $f_k$ and a
priori depending on the geodesic $\gamma$ converging toward $\xi$.
In any case, we have
%
%
\begin{equation}
\label{Derriennic} \| T_{j_0}\cdots T_{j_{k-1} } f_k -
f_\infty\|_t \leq C \tau^k \quad\mbox{and}\quad \Phi(
\xi) = \ln\frac{ \langle\a,f_\infty\rangle} { \langle\a_1,
f_\infty\rangle}.
\end{equation}

Consider now two points $\xi, \eta\in\partial G$ such that $\rho
(\xi, \eta) < a^{-n -C}$. Then there is a geodesic $\gamma$
converging to $\xi$ and a sequence $\{y_\ell\}_{\ell\geq1}$ going
to $\eta$ such that for $\ell, m$ large enough, $(\gamma(m), y_\ell
) >n$. For fixed $x$ and $K = 4M+|x|$, consider the same family
$U_0^-\subset U_0 \subset\cdots\subset U_k^- \subset U_k$ such that
for all $j = 1, \ldots, k-1$,
$U_j^-\subset U_j \subset U_{j+1}^- \subset U_{j+1}$ is an obstacle for
$\gamma\circ\s^{2jM+K}$. We have:
%
%
\begin{lem}\label{catchy} Assume $2kM < n -K -4r -22 \delta$ and
$\ell$ large enough. Then, $y_\ell\notin U_k$.
\end{lem}
\begin{pf} Choose $\ell$ large enough that $\lim_{m \to\infty}
(\gamma(m), y_\ell) >n $, and we choose a geodesic $[y_\ell, \xi]$
such that $(y_\ell, \xi) > n$. By definition of $U_j$, we have to
show that $d(y_\ell, \gamma(2(j-1)M+K)) \geq d(y_\ell, \gamma(2jM
+K + 4r))$ for $2jM +K+r+22\delta<n$.
By continuity, there is a point $s_0$ where the function $s \mapsto
d(y_\ell, \gamma(s)) $ attains its minimum. We are going to show that
$s_0 \geq n - 12 \delta$. By $8 \delta$ convexity of $s \mapsto
d(y_\ell, \gamma(s))$ (\cite{GH}, Proposition 25, page 45), this
proves the claim.\footnote{Indeed, since $2jM +K+r+22\delta<n$, $
\gamma(2jM +K + 4r +10\delta) $ lies between $ \gamma(2(j-1)M +K)$
and $s_0$ and thus, by $8\delta$ convexity of the distance, $ d(y_\ell
, \gamma(2jM +K + 4r +10\delta)) \leq d(y_\ell, \gamma
(2(j-1)M+K))+8\delta$ [recall that $s_0$ achieves the minimum of
$d(y_\ell, \gamma(s))$]. The inequality follows by writing the
$\delta$-hyperbolicity relation (\ref{hyp}) with $x = y_\ell, y =
\gamma(2jM +K + 4r ), z = \gamma(2jM +K + 4r +10\delta)$ and $w =
\gamma(2(j-1) M +K)$.}

By continuity, there is a point $s_1$ such that $d(\gamma(s_1),
[\gamma(0), y_\ell] ) = d(\gamma(s_1),\break [ y_\ell, \xi] ) \leq
4\delta$. On the one hand,
\[
s_1 \geq d\bigl(\gamma(0), [y_\ell, \xi]\bigr) - 4\delta
\geq n -4\delta
\]
[recall that $(\xi, y_\ell) >n$]. On the other hand, we know that
\[
d\bigl(y_\ell, \gamma(s_1) \bigr) \leq\bigl(\gamma(0),
\xi\bigr)_{y_\ell} + 8 \delta\leq d\bigl(y_\ell,
\gamma(s_0)\bigr) +8\delta
\]
(see the proof of Lemma 22.4 in~\cite{W}). It follows that $s_0 \geq
s_1 - 8\delta\geq n-12\delta$.
\end{pf}

We have that $\Phi(\xi) - \Phi(\eta) = \lim_{x_m \to\xi, y_\ell
\to\eta} \ln(\frac{u(e,x_m)}{u(x,x_m)} \frac{u(x,y_\ell
)}{u(e,y_\ell)} )$. With the above notation, assume that $k$ is
such that $2kM < n-K -4r-22\delta$. If $\ell$ and $m$ are large
enough, $y_\ell, \gamma(m) \notin U_k $ and
\[
\Phi(\xi) - \Phi(\eta) = \lim_{x_m \to\xi, y_\ell\to\eta} \ln\frac{
\langle\b, T_{j_0}\cdots T_{j_{k-1} } f_k(\gamma(m))
\rangle} { \langle\a, T_{j_0}\cdots T_{j_{k-1}} f_k (\gamma(m))
\rangle}
\frac{ \langle\a, T_{j_0}\cdots T_{j_{k-1} } f_k(y_\ell
) \rangle} { \langle\b, T_{j_0}\cdots T_{j_{k-1}} f_k(y_\ell)
\rangle}.
\]
Since as above, we have $\vt( T_{j_0}\cdots T_{j_{k-1} } f_k(\gamma
(m)), T_{j_0}\cdots T_{j_{k-1} } f_k(y_\ell) ) < C \tau^k$, and $\a
, \b$ take a finite number of values, we have
\[
\bigl|\Phi(\xi) - \Phi(\eta)\bigr| \leq C \tau^k \leq C \rho_0^n
\]
for a new constant $C$ and $\rho_0= \tau^{1/2M}$. This shows that for
all $x \in G$, the function $\xi\mapsto-\ln K_x (\xi) $ is H\"older
continuous on $\partial G$. Moreover, the H\"older exponent $|{\ln\rho
_0}|/ \ln a $ and the H\"older constant $C$ are uniform on a
neighborhood of $p$ in $\M(F)$.\vadjust{\goodbreak}

Let us choose $ \k<1, \k< -\frac{\ln\rho_0}{2\ln a}
$, and consider the space $\G_\k$ of functions $\phi$ on $\partial
G$ such that there is a constant $C_\k$ with the property that $|\phi
(\xi) - \phi(\eta)| \leq C_\k(d(\xi, \eta))^\k$. For $\phi\in
\G_\k$, denote $\| \phi\|_\k$ the best constant $C_\k$ in this
definition. The space $\G_\k$ is a Banach space for the norm $ \|
\phi\|:= \| \phi\|_\k+ \max_{\partial G} |\phi|$. In this
subsection, we showed that for $p \in\M(F) $, $x \in G$, there exist
$\k>0$ and a neighborhood $\O_p$ of $p$ in $\M(F)$ such that for $p'
\in\O_p$, the function $\Phi_{p'}( \xi) =
- \ln K_{\xi} (x)$ belongs to $\G_\k$ and that the mapping $p'
\mapsto\Phi_{p'} $ is bounded from $\O_p $ into $\G_\k$.\vspace*{-2pt}

\section{The Martin kernel depends regularly on $p$}\vspace*{-2pt}
%
%
\begin{prop}\label{LipMartin} Fix $x \in G$. For all $p \in\M(F)$,
there exist $\k>0$ and a neighborhood $\O_p$ of $p$ in $\M(F)$ such
that the mapping $p \mapsto\Phi(\xi) = - \ln K_{\xi} (x)$ is
Lipschitz continuous from $\O_p$ into $\G_\k$.\vspace*{-2pt}
\end{prop}
\begin{pf} Let $p \in\M(F)$ and choose $\k= \k(p) $ given by
Section~\ref{Martinreg}. We have to find a neighborhood $\O$ of $p$
in $\M(F)$ and a constant $C$ such that, for $p'\in\O$,
\[
\| \Phi_p - \Phi_{p'} \| = \max_\xi\bigl|
\Phi_p (\xi)- \Phi_{p'} (\xi) \bigr| + \|\Phi_p -
\Phi_{p'} \|_\k\leq C \vt\bigl(p,p'\bigr),
\]
where, for convenience, we use on $\M(F)$ the already defined
projective distance on $\R^F$. We treat the two terms separately.\vspace*{-2pt}
\begin{Claim}\label{claim1} $\max_\xi|\Phi_p (\xi)- \Phi_{p'} (\xi) | \leq C
\vt(p,p')$.\vspace*{-2pt}
\end{Claim}

Choose the geodesic $\gamma$ converging to $\xi$. Applying Section
\ref{Martinreg} and (\ref{Derriennic}), there are vectors $f_\infty
(p), f_\infty(p') \in\ell^t(V_0)$ such that
\[
\bigl|\Phi_p (\xi)- \Phi_{p'} (\xi) \bigr| = \biggl| \ln
\frac{ \langle\a
(p),f_\infty(p)\rangle} { \langle\a(p'), f_\infty(p')\rangle} \frac{
\langle\b(p'),f_\infty(p') \rangle} { \langle\b(p),
f_\infty(p)\rangle} \biggr|.
\]
By Proposition~\ref{cont}, we make an error of order $C\vt(p,p')$
when replacing $\b(p')$ by $\b(p) $ and $\a(p')$ by $\a(p) $.
The remaining term is
%
%
\begin{equation}
\label{uniformerror} \lim_k \biggl|\ln\frac{
\langle\a, T_{j_0}\cdots T_{j_{k-1} } f_k\rangle} { \langle\a,
T'_{j_0}\cdots T'_{j_{k-1}} f_k \rangle} \frac{ \langle\b,
T'_{j_0}\cdots T'_{j_{k-1} } f_k \rangle} { \langle\b,
T_{j_0}\cdots T_{j_{k-1}} f_k \rangle}
\biggr|,
\end{equation}
where $T_{j_s}$ is the projective action of $A_{V_s}^{V_{s+1}}(p)$,
$T'_{j_s}$ the projective action of $A_{V_s}^{V_{s+1}}(p')$\vspace*{2pt} and we have
chosen once for all $f_k \in\ell^t (V_k)$, independent of $p \in\O$.

We have
\begin{eqnarray*}
&& \vt\bigl(T_{j_0}\cdots T_{j_{k-1} } f_k,T'_{j_0}
\cdots T'_{j_{k-1}
} f_k \bigr)
\\
&&\qquad\leq \sum_{i = 1}^{k-1} \vt
\bigl(T_{j_0}\cdots T_{j_{i-1} } T'_{j_i}
\cdots T'_{j_{k-1} } f_k, T_{j_0}\cdots
T_{j_{i} } T'_{j_{i+1}}\cdots T'_{j_{k-1} }
f_k \bigr)
\\
&&\qquad\leq \sum_{i = 1}^{k-1}\tau^{i-1}
\vt\bigl( T'_{j_i}\cdots T'_{j_{k-1} }
f_k, T_{j_i}T'_{j_{i+1}}\cdots
T'_{j_{k-1} } f_k \bigr),
\end{eqnarray*}
where we used (\ref{exp}) to write the last line. If the neighborhood
$\O$ is relatively compact in $\M(F)$, all points $T'_{j_{i+1}}\cdots
T'_{j_{k-1} } f_k $ are in a common bounded
subset of $\DD^t_{V_{j_{i+1}}}$. By Proposition~\ref{Acont}, there is
a constant $C$ and a neighborhood $\O$ such that for $p' \in\O$, $i
= 1, \ldots, k-1$,
\[
\vt\bigl( T'_{j_i}T'_{j_{i+1}}
\cdots T'_{j_{k-1} } f_k, T_{j_i}T'_{j_{i+1}}
\cdots T'_{j_{k-1} } f_k \bigr) \leq C\vt
\bigl(p,p'\bigr).
\]
Finally, we get that for all $k$, $ \vt
(T_{j_0}\cdots T_{j_{k-1} } f_k,T'_{j_0}\cdots T'_{j_{k-1} } f_k )
\leq\frac{C}{1-\tau} \vt(p,p')$. Reporting in (\ref{uniformerror})
proves Claim~\ref{claim1}.

\begin{Claim}\label{claim2}
$ \|\Phi_p - \Phi_{p'} \|_\k\leq C \vt(p,p')$.
\end{Claim}

Let $\xi, \eta\in\partial G$ be such that $\rho(\xi, \eta) <
a^{-n -C}$. We want to show that there is a constant $C$ and a
neighborhood $\O$, independent on $n$ such that, for $p' \in\O$,
\[
\bigl|\Phi_p (\xi)- \Phi_{p'} (\xi) -\Phi_p (
\eta)+ \Phi_{p'} (\eta)\bigr| \leq C a^{-\k n } \vt
\bigl(p,p'\bigr).
\]
Choose as before a geodesic $\gamma$ converging to $\xi$ and a
sequence $\{y_\ell\}_{\ell\geq1}$ going to $\eta$ such that for
$\ell, m$ large enough, $(\gamma(m), y_\ell) >n$. For fixed $x$ and
$K = 4M+|x|$, consider the same family $U_0^-\subset U_0 \subset\cdots
\subset U_k^- \subset U_k$ such that for all $j = 1, \ldots, k-1$,
$U_j^-\subset U_j \subset U_{j+1}^- \subset U_{j+1}$ is an obstacle for
$\gamma\circ\s^{2jM+K}$. By Lem\-ma~\ref{catchy}, for $\ell$ large
enough, $y_\ell\notin U_k $, and we may write $\Phi_p (\xi)- \Phi_{p'}
(\xi) -\Phi_p (\eta)+ \Phi_{p'} (\eta) $ as
%
%
\begin{eqnarray}
\label{long}
&&
\lim_{x_m \to\xi, y_\ell\to\eta} \ln\frac{ \langle\b,
T_{j_0}\cdots T_{j_{k-1} }g_k \rangle} {
\langle\a, T_{j_0}\cdots T_{j_{k-1}} g_k \rangle} \frac{ \langle
\a', T'_{j_0}\cdots T'_{j_{k-1} }g'_k \rangle} { \langle\b',
T'_{j_0}\cdots T'_{j_{k-1}} g'_k \rangle}
\nonumber\\[-8pt]\\[-8pt]
&&\qquad\hspace*{35pt}{}\times\frac{ \langle\a,
T_{j_0}\cdots T_{j_{k-1} } h_k \rangle} { \langle\b, T_{j_0}\cdots
T_{j_{k-1}} h_k \rangle} \frac{ \langle\b', T'_{j_0}\cdots
T'_{j_{k-1} }h'_k \rangle} { \langle\a', T'_{j_0}\cdots
T'_{j_{k-1}} h'_k \rangle},\nonumber
\end{eqnarray}
where $\a= \a(p), \a' = \a(p'), \b= \b(p), \b' = \b(p')$,
$T_{j_s}$ is the projective action of $A_{V_s}^{V_{s+1}}(p)$,
$T'_{j_s}$ the projective action of $A_{V_s}^{V_{s+1}}(p')$ and $g_k,
g'_k$ are $f_k(\gamma(m) )$ calculated with $p$ and $p'$, respectively,
$h_k, h'_k$ are $f_k(y_\ell)$ calculated with $p$ and $p'$.

Recall that $g_k = f_k (\gamma(m)) $ is the direction of $\omega
_v^{\gamma(m)} $ in $\ell^t(V_k)$. It can obtained by a series of
obstacles along $\gamma$ between $U_k $ and $\gamma(m) $. Let us show
that we can choose $m$ large enough (depending on $p'$) such that we
have $\vt(g_k, g'_k) \leq C \vt(p,p')$. Indeed,
\begin{eqnarray*}
\vt\bigl(g_k, g'_k\bigr) &=& \vt
\bigl(f_k \bigl(\gamma(m)\bigr), f'_k \bigl(
\gamma(m)\bigr)\bigr) \\
&=& \vt\bigl(T_{j_k}\cdots T_{j_{m-1} }
f_m,T'_{j_k}\cdots T'_{j_{m-1} }
f'_m \bigr).
\end{eqnarray*}
We have $\vt(f_m, f'_m) < C $ and for $m$ large enough,
\[
\vt\bigl(T'_{j_k}\cdots T'_{j_{m-1} }
f_m,T'_{j_k}\cdots T'_{j_{m-1} }
f'_m \bigr) < \tau^{m-k } C \leq\vt
\bigl(p,p'\bigr).
\]
By the same computation as in Claim~\ref{claim1}, we then have
\[
\vt\bigl(T_{j_k}\cdots T_{j_{m-1} } f_m,T'_{j_k}
\cdots T'_{j_{m-1} } f_m \bigr) \leq C \vt
\bigl(p,p'\bigr).
\]
Since $\vt(g_k, g'_k) \leq C \vt(p,p')$, using the contraction of the
$T_j$, we can replace $g'_k$ by $g_k$ in (\ref{long}) with an error
less than $C\tau^k \vt(p, p') < C \rho_0^n \vt(p,p')$. In the same
way, following obstacles along the geodesic between
$\gamma(n) $ and $y_\ell$, we have, for $\ell$ large enough, $\vt
(h_k, h'_k) \leq\vt(p,p')$, and we can replace $h'_k $ by $h_k$ in
(\ref{long}) with an error less than $ C \rho_0^n \vt(p,p') $.

Observe also that all terms $ \dot\a= \a/\|\a\|_s, \dot\b= \b/
\|\b\|_s$ belong to $\DD^s_{V_{j_0}}$. We may write, considering, for
instance, $\langle\a', T'_{j_0}\cdots T'_{j_{k-1} }g'_k \rangle$,
\begin{eqnarray*}
\frac{ \langle\a', T'_{j_0}\cdots T'_{j_{k-1}
}g'_k \rangle} { \langle\a, T'_{j_0}\cdots T'_{j_{k-1}} g'_k
\rangle} &=& \frac{ \langle\a', A'_{j_0}\cdots A'_{j_{k-1} }g'_k
\rangle} { \langle\a, A'_{j_0}\cdots A'_{j_{k-1}} g'_k \rangle}
\\
&=& \frac{ \langle(A'_{j_{k-1}})^\ast\cdots(A'_{j_0})^\ast\a',
g'_k \rangle} { \langle( A'_{j_{k-1}})^\ast\cdots(A'_{j_0})^\ast
\a, g'_k \rangle}
\\
&=& \frac{\| \a' \|_s} {\| \a\|_s } \frac{ \langle
(T'_{j_{k-1}})^\ast\cdots(T'_{j_0})^\ast\dot\a', g'_k \rangle} {
\langle( T'_{j_{k-1}})^\ast\cdots(T'_{j_0})^\ast\dot\a, g'_k
\rangle},
\end{eqnarray*}
where $(T'_j)^\ast$ denotes the projective action of $(A'_j)^\ast$
on $\DD^s_{V_j}$. Observe that if we replace $ \a' $ by $ \a$, $\b'$ by
$\b$ in (\ref{long}) and use the above equation and its
analogs, the ratios $\frac{\| \a' \|_s} {\| \a\|_s } $, $\frac{\|
\b'\|_s} {\| \b\|_s } $ cancel one another, and
using the contraction of the $(T'_j)^\ast$, we make an other error of
size at most $ C \rho_0^n \vt(p,p') $.

We find that, up to an error of size at most $ C \rho_0^n \vt(p,p')
$, the difference $\Phi_p (\xi)- \Phi_{p'} (\xi) -\Phi_p (\eta)+
\Phi_{p'} (\eta)$ is given by
\begin{eqnarray*}
&&
\lim_{x_m \to\xi, y_\ell\to\eta} \ln\frac{ \langle\dot\b,
A_{j_0}\cdots A_{j_{k-1} }g_k \rangle} { \langle\dot\b,
A'_{j_0}\cdots A'_{j_{k-1}} g_k \rangle} \frac{ \langle\dot\a,
A'_{j_0}\cdots A'_{j_{k-1} }g_k \rangle} { \langle\dot\a,
A_{j_0}\cdots A_{j_{k-1}} g_k \rangle} \\
&&\qquad\hspace*{35pt}{}\times\frac{ \langle\dot\a,
A_{j_0}\cdots A_{j_{k-1} } h_k \rangle} { \langle\dot\a,
A'_{j_0}\cdots A'_{j_{k-1}} h_k \rangle}
\frac{ \langle\dot\b,
A'_{j_0}\cdots A'_{j_{k-1} }h_k \rangle} { \langle\dot\b,
A_{j_0}\cdots A_{j_{k-1}} h_k \rangle},
\end{eqnarray*}
where we reordered the denominators to get a sum of four terms of the
form
\[
\pm\ln\frac{ \langle\a, A_{j_0}\cdots A_{j_{k-1} }g \rangle} {
\langle\a, A'_{j_0}\cdots A'_{j_{k-1}} g\rangle}
\]
with $\a\in\DD^s_{V_{j_0}}, g \in\DD^t_{V_{j_k}}$. We can arrange
each such term and write
\begin{eqnarray*}
& & \frac{ \langle\a, A_{j_0}\cdots A_{j_{k-1} }g \rangle} {
\langle\a, A'_{j_0}\cdots A'_{j_{k-1}} g\rangle}
\\
&&\qquad= \prod_{i=0}^{k-1} \frac{ \langle\a, A_{j_0}\cdots A_{j_{i-1}}
A_{j_i}A'_{j_{i+1}} \cdots A'_{j_{k-1} }g \rangle} { \langle\a,
A_{j_0}\cdots A_{j_{i-1}} A'_{j_i}A'_{j_{i+1}} \cdots A'_{j_{k-1}}
g\rangle}
\\
&&\qquad= \prod_{i=0}^{[k/2]} \frac{ \langle( A'_{j_{i-1}})^\ast\cdots
(A'_{j_0})^\ast\a, A'_{j_i} g_i \rangle} { \langle(
A'_{j_{i-1}})^\ast\cdots(A'_{j_0})^\ast\a, A_{j_i} g_i \rangle}
\times\prod_{i=[k/2]+1}^{k-1} \frac{\langle(A_{j_i})^\ast\a_i,
A'_{j_{i+1}}\cdots A'_{j_{k-1}}g \rangle}{\langle(A'_{j_i})^\ast\a_i,
A'_{j_{i+1}}\cdots A'_{j_{k-1}}g \rangle}
\\
&&\qquad= \prod_{i=0}^{[k/2]} \frac{ \langle(A_{j_i})^\ast\a_i, g_i
\rangle} { \langle(A'_{j_i})^\ast\a_i, g_i \rangle}
\times\prod_{i=[k/2]+1}^{k-1} \frac{\langle\a_i, A_{j_i} g_i \rangle
}{\langle\a_i, A'_{j_i} g_i \rangle},
\end{eqnarray*}
where $\a_i = (T_{j_{i-1}})^\ast\cdots(T_{j_0})^\ast\a$, $ g_i =
T'_{j_{i+1}}\cdots T'_{j_{k-1}} g$. Set $\b_i = (T_{j_{i-1}})^\ast
\cdots(T_{j_0})^\ast\b$, $ h_i = T'_{j_{i+1}}\cdots T'_{j_{k-1}} h$.
We are reduced to estimate
\begin{eqnarray*}
& & \prod_{i=0}^{[k/2]} \frac{ \langle(A_{j_i})^\ast\b_i, g_i
\rangle} { \langle(A'_{j_i})^\ast\b_i, g_i \rangle}
\frac{
\langle(A'_{j_i})^\ast\a_i, g_i \rangle} { \langle
(A_{j_i})^\ast\a_i, g_i \rangle}\frac{ \langle(A_{j_i})^\ast\a_i, h_i
\rangle} { \langle(A'_{j_i})^\ast\a_i, h_i \rangle
}\frac{ \langle(A'_{j_i})^\ast\b_i, h_i \rangle} { \langle
(A_{j_i})^\ast\b_i, h_i \rangle}
\\
&&\qquad{}\times \prod_{i=[k/2]+1}^{k-1}
\frac{\langle\b_i, A_{j_i} g_i
\rangle}{\langle\b_i, A'_{j_i} g_i \rangle} \frac{\langle\a_i,
A'_{j_i} g_i \rangle}{\langle\a_i, A_{j_i} g_i \rangle} \frac
{\langle\a_i, A_{j_i} h_i \rangle}{\langle\a_i, A'_{j_i} h_i
\rangle} \frac{\langle\b_i, A'_{j_i} h_i \rangle}{\langle\b_i,
A_{j_i} h_i \rangle}.
\end{eqnarray*}

Since, $g_i, h_i $ remain in a bounded part of the $\DD_{V^t}$ and $\a
_i, \b_i $ in a bounded part of the $\DD_{V^s}$, using Propositions
\ref{cont} and~\ref{cont*}, one gets a constant $C$ such that $\vt
(A_{j_i} g_i, A'_{j_i} g_i ), \vt(A_{j_i} h_i, A'_{j_i} h_i )$, $ \vt
((A_{j_i})^\ast\a_i, (A'_{j_i})^\ast\a_i )$ and $ \vt
((A_{j_i})^\ast\b_i,\break (A'_{j_i})^\ast\b_i) $ are all smaller than
$C\vt(p,p') $. Furthermore, using the contraction of $T_j$ and
$(T_j)^\ast$ (Proposition~\ref{Acones}) we see that
\[
\vt(\a_i, \b_i ) \leq C \tau^i,\qquad
\vt(g_i, h_i ) \leq C\tau^{k-i}.
\]
Moreover, all products in the formula are approximations of $\langle\a
, f_\infty\rangle$ and thus are uniformly bounded away from $0$. It
follows that for $i \leq k/2$,
\begin{eqnarray*}
&&\biggl| \ln\frac{ \langle(A_{j_i})^\ast\b_i, g_i \rangle\langle
(A'_{j_i})^\ast\b_i, h_i \rangle} { \langle(A'_{j_i})^\ast\b_i
, g_i \rangle\langle(A_{j_i})^\ast\b_i, h_i \rangle} \biggr| ,\\
&&\qquad \biggl| \ln\frac{
\langle(A'_{j_i})^\ast\a_i, g_i \rangle
\langle(A_{j_i})^\ast\a_i, h_i \rangle} { \langle(A_{j_i})^\ast
\a_i, g_i \rangle\langle(A'_{j_i})^\ast\a_i, h_i \rangle
} \biggr| \leq C \tau^{k-i }
\vt\bigl(p,p'\bigr)
\end{eqnarray*}
and for $i > k/2$,
\[
\biggl| \ln\frac{ \langle\b_i, A_{j_i} g_i \rangle\langle\a_i,
A'_{j_i} g_i \rangle} { \langle\b_i, A'_{j_i} g_i \rangle
\langle\a_i, A_{j_i} g_i \rangle} \biggr|, \biggl| \ln\frac{
\langle\a_i, A_{j_i}h_i \rangle\langle\b_i, A'_{j_i}h_i \rangle
} { \langle\a_i, A'_{j_i} h_i \rangle\langle\b_i, A_{j_i}h_i
\rangle} \biggr| \leq C \tau^{i }
\vt\bigl(p,p'\bigr),
\]
so that finally the main term of (\ref{long}) is estimated by
\[
\sum_{i=0}^{[k/2]} C\tau^{k-i} \vt
\bigl(p,p'\bigr) + \sum_{i=[k/2]+1}^{k-1}
C \tau^{i} \vt\bigl(p,p'\bigr) \leq C
\tau^{k/2} \vt\bigl(p,p'\bigr) \leq C \rho_0^{n/2}
\vt\bigl(p,p'\bigr).
\]
Claim~\ref{claim2} is proven (recall that $\k< -\frac{\ln\rho_0}{2\ln a} $ so
that $\rho^{1/2} < a^{-\k}$).
\end{pf}

\section{\texorpdfstring{Markov coding and regularity of $p^\infty$}{Markov coding and regularity of p infinity}}

In this section, we discuss the regularity of the mapping $p \mapsto
p^\infty$ from $\M(F)$ into the space $\G_\k^\ast$ of continuous
linear forms on $\G_\k$. By Theorem~\ref{stat}, $p^\infty$ is the
only $p$-stationary measure for the action of $\partial G$, and thus
depends continuously on $p$. In the case of the free group, $p^\infty$
appears as the eigenform for an isolated maximal eigenvalue of an
operator on $\G_\k$ (see~\cite{Le}, Chapter 4c)
and therefore depends real analytically on $p$. This argument does not
seem to work in all the generality of a hyperbolic group, and we are
going to use the Markov representation of the boundary which was
described by M.~Coornaert and A. Papadopoulos in~\cite{CP}.

\subsection{Markov coding}\label{Markov}

Following~\cite{CP}, we call \textit{horofunctions} any integer valued
function on $G$ such that, for all $\l\leq h(x)$, the distance of a
point $x$ to the set $h^{-1} (\l) $ is given by $h(x) - \l$. Two
horofunctions are said to be equivalent if they differ by a constant.
Let $\Phi_0 $ be the set of classes of horofunctions. Equipped with
the topology of uniform convergence on finite subsets of $G$, the space
$\Phi_0$ is a compact metric space. $G$ acts naturally on $\Phi_0$.
The Busemann boundary $\partial_B G $ is a $G$-invariant subset of
$\Phi_0$. For each horofunction $h$, sequences $\{x_n\}_{n\geq1}$
such that
\[
d(x_n, x_{n+1}) = h(x_n) - h
(x_{n+1} ) = 1
\]
converge to a common point in $\partial G$, the point at infinity of
$h$. Two equivalent horofunctions have the same point at infinity. The
mapping $\pi\dvtx\Phi_0 \to\partial G$ which associates to a class of
horofunctions its point at infinity is continuous, surjective,
$G$-equivariant and uniformly finite-to-one.
Fix an arbitrary total order relation on the set of generators $S$.
Define a map $\a\dvtx\Phi_0 \to\Phi_0$ by setting, for a class $ \v
=[h ] \in\Phi_0$, $ \a(\v) = a^{-1} \v$, where $a = a(\v) $ is
the smallest element in $S$ satisfying $h(e) - h(a) = 1$. In~\cite{CP}
is proven:
%
%
\begin{theo}[(\cite{CP})]\label{subshift} The dynamical system $(\Phi_0,
\a)$ is topologically conjugate to a subshift of finite type.
\end{theo}

We assume, as we may, that the number $R_0$ used in the construction of
\cite{CP} satisfies $R_0 >r$. In order to fix notation, let $(\Si, \s
) $ be the subshift of finite type of Theorem~\ref{subshift}. That is,
there is a finite alphabet $Z$ and a $Z\times Z$ matrix $A$ with
entries $0$ or $1$ such that $\Si$ is the set of sequences $\underline
z =\{z_n\}_{n \geq0} $ such that for all~$n$, $A_{z_n, z_{n+1}} = 1 $
and $\s$ is the left shift on $\Si$. We can decompose $\Si$ into
transitive components. Namely, there is a partition of the alphabet $Z$
into the disjoint union of $Z_j, j = 0, \ldots, K$ in such a way that
for $j = 1, \ldots, K$, $\Si_j:= \{ \underline z, z_0 \in Z_j \} $
is a $\s$-invariant transitive subshift of finite type and $\bigcup
_{j=1}^K \Si_j $ is the $\omega$-limit set of~$\Si$. By
construction, $G$-invariant closed subsets of $\Si$ are unions of $\Si
_j$ for some $j \in\{1, \ldots, K \}$. We denote such $G$-invariant
subsets by $\Si_J$, where $J$ is the corresponding subset of $\{ 1,
\ldots, K\}$. In particular, the supports of stationary measures on
$\overline G$ are subsets of $\partial_B G$ which are identified with such
$\Si_J$.

For $\chi>0$ consider the space $\G_\chi$ of functions $\phi$ on
$\Si$ such that there is a constant $C_\chi$ with the property that,
if the points $\underline z $ and $\underline z'$ have the same first
$n $ coordinates, then $| \phi(\underline z ) - \phi(\underline z')|
< C_\chi\chi^n. $ For $\phi\in\G_\chi$, denote $\| \phi\|_\chi$
the best constant $C_\chi$ in this definition. The space $\G_\chi$
is a Banach space for the norm $ \| \phi\|:= \| \phi\|_\chi+ \max_{\Si}
|\phi|$. Identifying $\Si$ with $\Phi_0$, we still write
$\pi\dvtx \Si\to\partial G$ the mapping which associates to $\underline
z \in\Si$ the point at infinity of the class of horofunctions
represented by $\underline z$.
%
%
\begin{prop} \label{symbolic} The mapping $\pi\dvtx \Si\to\partial G$
is H\"older continuous.
\end{prop}
\begin{pf} Let $\underline z$ and $\underline{z'}$ be two elements of
$\Si$ such that $z_i = z'_i $ for $1\leq i \leq n$. Denote $h$ and
$h'$ the corresponding horofunctions with $h(e) = h'(e) = 0$. Let $\{
x_n \}_{n\geq0} $ be define inductively such that $x_0 = e$ and
$(x_{n-1})^{-1} x_n $ is the smallest element $a $ in $S$ such that
$h(x_{n-1}) - h(x_{n-1} a) = 1$. The sequence $\{x_n \}_{n\geq0} $ is
a geodesic and converges to $\pi(\underline z) $. By~\cite{CP}, Lemma
6.5, $h $ and $h' $ coincide on $N( \{x_0, \ldots, x_{n +L_0} \}, R_0
)$, where $L_0$ and $R_0$ has been chosen as in~\cite{CP}, page 439.
In particular, if one associates $\{x'_n \}_{n\geq0} $ similarly to
$h'$, the sequence $\{x'_n \}_{n\geq0} $ is a geodesic which converges
to $\pi(\underline z') $, and we have $x_k = x'_k$ for $0 \leq k \leq
n + L_0$. It follows that for all $m,m' > n +L_0$,
\[
(x_m, x_{m'})_e = n+ L_0 +
(x_m, x_{m'} )_{x_{n+L_0}} \geq n+ L_0
\geq n.
\]
Therefore
\[
\bigl(\pi(\underline z), \pi\bigl(\underline{z'} \bigr)\bigr)_e \geq\liminf
_{m,m'} (x_m, x_{m'})_e \geq n\quad \mbox{and} \quad
\rho\bigl(\pi(\underline z), \pi
\bigl(\underline{z'} \bigr)\bigr) \leq e^{-an +c_1}.
\]
\upqed\end{pf}
In the same way, we have:
%
%
\begin{prop}\label{regularityh} Let $x$ be fixed in $G$ with $|x| <
R_0$. Then the mapping $\underline z \mapsto h_{\underline z} (x) $
depends only on
the first coordinate in $\Si$, where $h_{\underline z} $ is the horofunction
representing $\underline z$ in Theorem~\ref{subshift}.
\end{prop}
\begin{pf} As above, if $z_0 = z'_0 $ and $h,h'$ are the corresponding
horofunctions with $h(e) = h'(e) = 0 $, $h$ and $h'$ coincide on $N(e,
R_0) \supset\{x\} $.
\end{pf}

Let $\nu$ be a stationary probability measure on $\Phi_0$. By
equivariance of $\pi$, the measure $\pi_\ast\nu$ is stationary on
$\partial G$ and, by Theorem~\ref{stat}, we have $\pi_\ast\nu=
p^\infty$. Actually, there is a more precise result:
%
%
\begin{prop}Let $\nu$ be a stationary measure on $\Phi_0$. Then, for
$\nu$-a.e. $\v\in\Phi_0$, all $x$,
%
%
\begin{equation}
\label{density2} \frac{dx_\ast\nu}{d\nu}(\v) = K_{\pi(\v)} (x).
\end{equation}
\end{prop}
\begin{pf}
Since the mapping $\pi\dvtx \Phi_0 \to\partial F$ is $G$-equivariant and
finite-to-one, the measure $\nu$ can be written as
\[
\int\psi(\v) \,d\nu(\v) = \int\biggl(\sum_{\v\dvtx\pi(\v) =
\xi}
\psi(\v) a(\v) \biggr) \,dp^\infty(\xi),
\]
where $a$ is a nonnegative measurable function on $\Phi_0 $ such that
$\sum_{\v\dvtx\pi(\v) = \xi} a(\v) = 1$ for $p^\infty$-a.e. $\xi$.
Moreover, since $\partial G$ is a Poisson boundary for the random walk
(\cite{K}, Theorem 7.6), the conditional measures $a(\v) $ has to
satisfy $a(x\v) = a(\v)$ $p^\infty$-a.s. (\cite{KV}, Theorem 3.2).
Formula (\ref{density2}) for the density then follows from formula
(\ref{density1}).
\end{pf}
Identifying $\Phi_0 $ with $\Si$, we see that, for $\underline z \in
\Si, $ $\s^{-1} \underline z$ is given by some $a\underline z$, where
$a$ is one of the generators. We can describe the restriction of a
stationary measure to $\Si_j$. More precisely, we have:
%
%
\begin{prop}\label{thermo} For each $j = 1, \ldots, K$, there is a
unique probability measure $\nu_j$ such that any $p$-stationary
measure on $\Si$ has the restriction to $\Si_j $ proportional to $\nu
_j$. Moreover, for all $p \in\M(F)$, there exist $\chi>0$ and a
neighborhood $\O_p$ of $p$ in $\M(F)$ such that the mapping $p
\mapsto\nu_j $ is Lipschitz continuous from $\O_p$ to $\G_\chi^\ast(\Si_j)$.
\end{prop}
\begin{pf}
Consider a $p$-stationary probability measure on $\Si$ that has a
nonzero restriction to $\Si_j$. Let $\nu_j $ be this (normalized)
restriction. By (\ref{density2}), for all $x$ such that $x^{-1} \Si_j
= \Si_j $, we have $ \frac{dx_\ast\nu_j}{d\nu_j }(\underline z) =
K_{\pi(\underline z)} (x)$. We shall show that there is a unique probability
measure on $\Si_j$ satisfying $ \frac{d\s_\ast\nu_j}{d\nu_j }
(\underline z) = K_{\pi(\underline z)} (z_0) $ and that
it depends Lipschitz continuously on $p$ as an element of $\G_\chi^\ast
$ for some suitable $\chi$.

We use thermodynamical formalism on the transitive subshift of finite
type $\Si_j$. For $\chi< 1$ and $\phi\in\G_\chi$ with real
values, we define the transfer operator $\LL_\phi$ on $\G_\chi(\Si_j)$ by
\[
\LL_\phi\psi(\xi):= \sum_{\eta\in\s^{-1}\xi}
e^{\phi
(\eta) } \psi(\eta).
\]
Then, $\LL_\phi$ is a bounded operator in $\G_\chi$. Ruelle's
transfer operator theorem (see~\cite{Bo}, Theorem\vadjust{\goodbreak} 1.7, and~\cite{R},
Proposition 5.24) applies to $\LL_\phi$, and there exists a number
$P(\phi) $ and a linear functional $N_\phi$ on $\G_\chi$ such
that
the operator $\LL_\phi^\ast$ on $(\G_\chi)^\ast$ satisfies $\LL_\phi
^\ast N_\phi= e^{P(\phi)} N_\phi$. The functional $N_\phi$
extends to a probability measure on $\Si_j $ and is the only
eigenvector of $\LL^\ast_\phi$ with that property.
Moreover, $\phi\mapsto\LL_\phi$ is a real analytic map from $\G_\chi$
to the space of linear operators on $\G_\chi$ (\cite{R},
page 91). Consequently, the mapping $\phi\mapsto N_\phi$ is real
analytic\vspace*{1pt} from $\G_\chi$ into the dual space $\G_\chi^\ast$; see,
for example,~\cite{C}, Corollary 4.6. For $p \in\M(F)$, define $\phi_p
(\underline z) = \ln K_{\pi(\underline z)} (z_0)$. By Propositions
\ref{LipMartin} and~\ref{symbolic}, we can choose $\chi$ such that
the mapping $p \mapsto\phi_p $ is Lipschitz continuous from a
neighborhood $\O_p$ of $p$ in $\M(F)$ into the space $\G_\chi$. It
follows that the mapping $p \mapsto N_{\phi_p} $ is Lipschitz
continuous from $\O_p$ into $\G_\chi^\ast$.

From the relation $ \frac{d\s_\ast\nu_j}{d\nu_j } (\underline z) =
K_{\pi(\underline z)} (z_0) $, we know that $\nu_j $ is invariant
under $\LL_{\phi_p}^\ast$. This shows that $\nu_j $ is the only
probability measure satisfying this relation, that $P(\phi_p) = 0 $
and that $\nu_j $ extends $N_{\phi_p} $.
\end{pf}

Let $\Si_J $ be a minimal, closed $G$-invariant subset of $\Si$. We
know that $\Si_J $ is a finite union of transitive subshifts of finite
type. We have:
%
%
\begin{cor}\label{final} For $p \in\M(F)$, there is a unique
$p$-stationary probability measure $\nu_J (p) $ on $\Si_J$. There is
a $\chi$ and a neighborhood $\O$ of $p$ such that the mapping $p
\mapsto\nu_J(p) $ is Lipschitz continuous from $\O$ into $\G_\chi^\ast
(\Si_J)$.
\end{cor}
\begin{pf} Let $\nu_J$ be a $p$ stationary measure on $\Si_J$. We
know by Proposition~\ref{thermo} that the conditional measures on the
transitive subsubshifts are unique and Lipschitz continuous from $\O$
into $\G_\chi^\ast(\Si_k)$. We have to show that the $\nu_J(\Si_k) $
are well determined and Lipschitz continuous in $p$. Write again
equation (\ref{density2}), but now for elements $x \in G$ that
exchange the $\Si_k $ within $\Si_J$ and write that $\sum_k \nu_J(\Si
_k) = 1 $. We find that the $\nu_J(\Si_k) $ are given by a
system of linear equations. By Propositions~\ref{LipMartin} and
\ref{thermo}, we know that the coefficients of this linear system are
Lipschitz continuous on $\O$. We know that there is a solution, and
that it is unique, since otherwise there would be a whole line of
solutions, in particular one which would give $\nu_J (\Si_k) =0 $ for
some $k$ and this is impossible. Then the unique solution is Lipschitz
continuous.
\end{pf}

\section{\texorpdfstring{Proof of Theorem \protect\ref{main}}{Proof of Theorem 1.1}}

Choose $\chi$ small enough and $\O$ a neighborhood of $p$ in $\M(F)$
such that Proposition~\ref{LipMartin} and Corollary~\ref{final}
apply: the mappings $ p \mapsto\ln K_{\pi(\underline z)} (x) $ and $p
\mapsto\nu_J$ are Lipschitz continuous from $\O$ into, respectively,
$\G_\chi(\Si)$ and $\G_\chi^\ast(\Si_J)$. Then, by definition
(\ref{entr}), the function $p \mapsto h_p(\Si_J, \nu) $ is Lipschitz
continuous on $\O$. By (\ref{entropy}) and (\ref{entropy2}), the
function $h_p$ is the maximum of a finite number of Lipschitz
continuous functions on $\O$; this proves the entropy part of
Theorem~\ref{main}.

For the escape rate part, recall that the Busemann boundary $\partial_B
G$ is made of horofunctions so that it can be identified\vadjust{\goodbreak} with a
$G$-invariant subset of $\Si$. Stationary measures on $\partial_B G$
are therefore convex combinations of the $\nu_{J'}$, where $J'$ are
such that $\nu_{J'} (\partial_B G) = 1$. Formula (\ref{escape}) yields
$\ell_p^S = \max_{J'} \{ \sum_{x \in F} ( \int_{\Si_{J'}} h(x^{-1})
\,d\nu_{J'}(h) )\* p(x) \}$. By Proposition~\ref{regularityh}, for a
fixed $x \in F$ the function $h(x)$ is in $\G_\chi(\Si_j)$ for all
$\chi$. Therefore, Corollary~\ref{final} implies that each one of the
functions $\int_{\Si_{J'}} h(x^{-1}) \,d\nu_{J'}(h)$ is Lipschitz
continuous on $\O$. This achieves the proof of Theorem~\ref{main}
because the function $p\mapsto\ell_p^S$ is also\vspace*{1pt} written as
the maximum of a finite number of Lipschitz continuous functions on
$\O$.



\printaddresses

\end{document}